\documentclass[letterpaper,11pt]{amsart}

\usepackage{amsmath}
\usepackage{amssymb}
\usepackage{amsthm}

\newtheorem{theorem}{Theorem}

\newtheorem{lemma}[theorem]{Lemma}

\newtheorem{conj}[theorem]{Conjecture}

\newcommand{\CC}{\mathbb{C}}
\newcommand{\EE}{\mathbb{E}}
\newcommand{\JJ}{\mathbb{J}}
\newcommand{\KK}{\mathbb{K}}
\newcommand{\NN}{\mathbb{N}}
\newcommand{\RR}{\mathbb{R}}
\newcommand{\VV}{\mathbb{V}}
\newcommand{\ZZ}{\mathbb{Z}}

\newcommand{\cF}{\mathcal{F}}
\newcommand{\cN}{\mathcal{N}}

\newcommand{\be}{\mathbf{e}}
\newcommand{\br}{\mathbf{r}}
\newcommand{\bs}{\mathbf{s}}

\newcommand{\bx}{\mathbf{x}}

\newcommand{\hz}{\hat{0}}

\DeclareMathOperator{\Res}{Res}

\DeclareMathOperator{\cov}{cov}
\renewcommand{\thefootnote}{\fnsymbol{footnote}}

\begin{document}
\title[HyperMatrix Spectra and the Product Formula]{Computing Hypermatrix Spectra with the Poisson Product Formula}
\maketitle
\begin{center} Joshua Cooper\footnote{\noindent cooper@math.sc.edu, This work was funded in part by NSF grant DMS-1001370.} and Aaron Dutle\footnote{\noindent dutle@mailbox.sc.edu, Corresponding author. } \\ Department of Mathematics\\ University of South Carolina \\1523 Greene St.\\ Columbia, SC, 29208\let\thefootnote\relax \let\thefootnote\relax
\end{center}

\begin{abstract}
We compute the spectrum of the ``all ones'' hypermatrix using the Poisson product formula.  This computation includes a complete description of the eigenvalues' multiplicities, a seemingly elusive aspect of the spectral theory of tensors.  We also give a general distributional picture of the spectrum as a point-set in the complex plane, and use our techniques to analyze the spectrum of ``sunflower hypergraphs'', a class that has played a prominent role in extremal hypergraph theory.
\end{abstract}

\section{Introduction}

There are few extant techniques that allow one to determine the characteristic polynomial of a symmetric hypermatrix in the sense of Qi (\cite{Qi_2005}).  Indeed, the lack of a simple formula for the symmetric hyperdeterminant makes many questions about hypermatrices much more difficult than the same questions for matrices.  In particular, we know of no simple way of determining the algebraic multiplicities of a hypermatrix's eigenvalues.

The characteristic polynomial and the symmetric hyperdeterminant are both defined using polynomial resultants, and so finding better methods of computing the resultant is one way to approach this problem. There is a tool for inductively computing resultants, the so-called  ``Poisson product formula'', which is given in \cite{Cox_Little_OShea_2005}.  Presently, we show that it can be used to compute the characteristic polynomial for some types of hypermatrices, thereby completely describing their spectra.  The statement of the product formula (from \cite{Cox_Little_OShea_2005}) is as follows.

\begin{theorem}
Let \( F_0, F_1, \ldots,  F_n \) be homogeneous polynomials of respective degrees \( d_0, \ldots, d_n \) in \( K[x_0, \ldots, x_n] \)  where \( K \) is an algebraically closed field. For \( 0\leq i\leq n, \) let \( \overline{F_i} \) be the homogeneous polynomial in \( K[x_1, \ldots, x_n] \) obtained by substituting \( x_0=0 \) in \( F_i \), and let \( f_i \) be the polynomial in \( K[x_1, \ldots, x_n] \) obtained by substituting \( x_0=1 \) in \( F_i. \) Let \( V \) be the set of simultaneous zeros of the system  of polynomials \( \{f_1, f_2, \ldots, f_n\}, \) that is, \( V \) is the affine variety defined by the polynomials.   If \( \Res(\overline{F_1}, \overline{F_2}, \ldots, \overline{F_n}) \neq 0 \), then \( V \) is a zero-dimensional variety (a finite set of points), and

\[\Res(F_0, F_1, \ldots, F_n) =
\Res(\overline{F_1}, \overline{F_2}, \ldots, \overline{F_n})^{d_0}\prod_{p\in V} f_n(p)^{m(p)}\]
where \( m(p) \) is the \emph{multiplicity} of a point \( p\in V. \)  \end{theorem}

A few notes are in order. In our case, our field will be \( K = \overline{\mathbb{C}[\lambda]} \), the algebraic closure of the (fraction field of) polynomials in an indeterminant \( \lambda \) with complex coefficients.  All computations in the sequel will concern characteristic polynomials of hypermatrices with complex entries, which are always monic (more importantly, nonzero), so the condition \( \Res(\overline{F_1}, \overline{F_2}, \ldots, \overline{F_n}) \neq 0 \) always holds.  The \emph{multiplicity} referred to above is defined in \cite{Cox_Little_OShea_2005}, but will be inconsequential in our uses. In each case, we show that every point of the variety has multiplicity at most one, and so identically one. Indeed, the variety \( V \) referred to above has at most \( \prod_{i=1}^n \deg(f_i) \) points in it by Bezout's theorem, and in our applications we exhibit exactly this many distinct points.

\section{The All-Ones Hypermatrix}

Fix \( m \geq 2 \).  Let \( \JJ_n^m \) denote the order \( m \), dimension \( n \) hypermatrix consisting of all ones.  We write \( \hat{0} \) and \( \hat{1} \) for the all-zeros and all-ones vectors of dimension \( n \), respectively.  For \( i \in [n] \), define
\[
F_i = \lambda x_i^{m-1} - \left ( \sum_{j=1}^n x_j \right )^{m-1}.
\]
Evidently, the characteristic polynomial of \( \JJ_n^m \) is \( \phi_{\JJ^m_n}(\lambda) = \phi(\lambda) = \Res(F_1,\ldots,F_n) \), a polynomial in \(\lambda\) of degree \(n (m-1)^{n-1}\).  First, we consider the spectrum of \( \JJ_n^m \) as a set, i.e., \( S = \{\lambda : \phi(\lambda)=0\} \), or alternatively,
\[
S = \left \{ \lambda : \exists \bx \forall i \in [n] (F_i(\bx) = 0) \right \}.
\]
Note that \( \bx = (1,-1,0,0,\ldots,0) \) is a solution of the system \( \cF = \{F_i = 0\}_{i=1}^n \) for \( \lambda = 0 \), so \( 0 \in S \).  Now, assume \( \lambda \neq 0 \), and let \( \bx \) be an eigenvector of \( \lambda \).  Note that \( \bx \) has all nonzero entries, since otherwise \( x_k = 0 \) would imply
\[
\lambda x_k^{m-1} = 0 = \left ( \sum_{j=1}^n x_j \right )^{m-1},
\]
so that \( \lambda x_i^{m-1} = 0 \), i.e., \( x_i = 0 \), for all \( i \in [n] \), whence \( \bx = \hz \), a contradiction.

Next, note that
\begin{equation} \label{eq5}
x_i^{m-1} = \frac{\left ( \sum_{j=1}^n x_j \right )^{m-1}}{\lambda}
\end{equation}
for all \( i \in [n] \).  Since the right-hand side of (\ref{eq5}) is independent of \( i \), we have that \( x_i^{m-1} = x_j^{m-1} \) for all \( i,j \in [n] \).  Since the \( F_i \) are homogeneous (in \( \bx \)), we may assume that \( x_1 = 1 \) and therefore \( x_i = \zeta_{m-1}^{a_i} \) for all \( i \in [n] \), where \( \zeta_q \) denotes a primitive \( q \)-th root of unity.  Therefore, each eigenvector \( \bx \) corresponding to a nonzero \( \lambda \) consists of \( r_i \) coordinates \( \zeta_{m-1}^i \) for some \( r_0,\ldots,r_{m-2} \in \NN \) with \( \sum_{i=0}^{m-2} r_i = n \).  It follows that \( \lambda = \sum_{i=0}^{m-2} r_i \zeta_{m-1}^i \), so we may conclude that
\[
S = \left \{ \left ( \br \cdot \hat{\zeta} \right )^{m-1} : \br \in \NN^{m-1}, \br \cdot \hat{1} = n \right \} \cup \{0\}.
\]
where \( \hat{\zeta} \) denotes the vector \( (1,\zeta_{m-1},\ldots,\zeta_{m-1}^{m-2}) \).  Although the above argument, which fully describes the {\em set} spectrum, gives no information about the multiplicities of the eigenvalues, it does yield a natural conjecture: the multiplicity of \( \xi = ( \br \cdot \hat{\zeta} )^{m-1} \) should be \( 1/(m-1) \) times the number of ways to choose \( s_0,\ldots,s_{m-2} \) so that
\[
\bs \cdot \hat{\zeta} = \xi^{1/(m-1)} \zeta_{m-1}^j
\]
for some \( j \in \{0,\ldots,m-2\} \).  Write \( \mu(\br) \) for this quantity.  In order to affirm our intuition, we apply the Poisson product formula.  To wit, let
\begin{eqnarray*}
\overline{F}_i = & F_i(x_1,x_2,\ldots,x_{n-1},0) \qquad & \text{for } i \in [n] \text{ and}\\
f_i = & F_i(x_1,x_2,\ldots,x_{n-1},1) \qquad & \text{for } i \in [n].
\end{eqnarray*}
Then, we define
\begin{equation} \label{eq6}
V = \VV(f_1,\ldots,f_{n-1}) \subseteq \KK^{n-1}
\end{equation}
the zero locus of the \( f_i \)'s, where \( \KK = \overline{\CC[\lambda]} \), the algebraic closure of \( \CC[\lambda] \).  Then the product formula states that
\begin{equation} \label{eq7}
\Res(F_1,\ldots,F_n) = \Res(\overline{F}_1,\ldots,\overline{F}_{n-1})^{\deg(F_n)} \prod_{\bx \in V} f_n(\bx)^{m(\bx)},
\end{equation}
where \( m(\bx) \) denotes the multiplicity of the point \( \bx \in V \).  As we shall see below, there are \( (m-1)^{n-1} \) distinct \( \bx \) in the product above -- the maximum possible number of such points by B\'{e}zout's Theorem applied to (\ref{eq6}), since \( \deg(f_i) = m-1 \) for all \( i \in [n-1] \).  Therefore, \( m(\bx) = 1 \) for all \( \bx \in V \).  Note also that the hypothesis of the product formula, that \( \psi = \Res(\overline{F}_1,\ldots,\overline{F}_{n-1}) \neq 0 \), holds here because \( \psi \) is a nonvanishing polynomial in \( \lambda \).

\begin{lemma} \label{lem:rootsofunity} Suppose \( p(x) = \alpha (x - \beta_1) \cdots (x - \beta_r) \) is a polynomial of degree \( r \) with lead coefficient \( \alpha \).  Then
\[
\prod_{j=0}^{r-1} p(\zeta_r^j x) = \alpha^r \prod_{j=0}^{r-1} (x^r - \beta_j^r)
\]
\end{lemma}
\begin{proof}
We can write
\begin{align*}
\prod_{j=0}^{r-1} p(\zeta_r^j x) &= \prod_{j=0}^{r-1} [\alpha (\zeta_r^j x - \beta_1) \cdots (\zeta_r^j x - \beta_r)] \\
&= \alpha^r \prod_{j=0}^{r-1} \zeta_r^{jr} (x - \zeta_r^{-j} \beta_1) \cdots (x - \zeta_r^{-j} \beta_r)] \\
&= \alpha^r (x^r - \beta_1^r) \cdots (x^r - \beta_r^r).
\end{align*}
\end{proof}

\begin{theorem} \label{thm:charpolyJJ} Let \( \phi_n(\lambda) \) denote the characteristic polynomial of \( \JJ_n^m \), \( n \geq 2 \).  Then
\[
\phi_n(\lambda) = \lambda^{(n-1) (m-1)^{n-1}} \prod_{\substack{\br \in \NN^{m-1} \\ \br \cdot \hat{1} = n}} \left ( \lambda - (\br \cdot \hat{\zeta})^{m-1} \right )^{\mu(\br)/(m-1)}.
\]
\end{theorem}

\begin{proof} We proceed by induction on \( n \) using (\ref{eq7}).  The base case \( n=1 \) is trivial.  For the inductive step, we write
\begin{align*}
\phi_n(\lambda) &= \Res(F_1,\ldots,F_n) \\
&= \Res(\overline{F}_1,\ldots,\overline{F}_n)^{m-1} \prod_{\bx \in V} f_n(\bx)^{m(\bx)} \\
&= \phi_{n-1}(\lambda)^{m-1} \prod_{\bx \in V} f_n(\bx)^{m(\bx)}.
\end{align*}
By the inductive hypothesis,
\[
\phi_{n-1}(\lambda) = \lambda^{(n-2) (m-1)^{n-2}} \prod_{\br \cdot \hat{1} = n-1} \left ( \lambda - (\br \cdot \hat{\zeta})^{m-1} \right )^{\mu(\br)/(m-1)}.
\]
Hence,
\begin{equation} \label{eq8}
\phi_{n}(\lambda) = \lambda^{(n-2) (m-1)^{n-1}} \prod_{\br \cdot \hat{1} = n-1} \left ( \lambda - (\br \cdot \hat{\zeta})^{m-1} \right )^{\mu(\br)} \prod_{\bx \in V} f_n(\bx)^{m(\bx)}.
\end{equation}
Note that
\[
V = \left \{ \bx \in \KK^{n-1} : \forall i \in n \left (\lambda x_i^{m-1} - \left (1 + \sum_{j=1}^{n-1} x_j \right )^{m-1} = 0 \right ) \right \}.
\]
We claim that \( x_i \neq 0 \) for any point \( \bx \in V \).  Indeed, suppose \( x_k = 0 \).  Then
\[
0 = \lambda x_k^{m-1} = \left ( 1 + \sum_{j=1}^{n-1} x_j \right )^{m-1}.
\]
Since this implies that, for any \( j \in [n-1] \),
\[
\lambda x_j^{m-1} = 0,
\]
it must be the case that \( x_j = 0 \) for all \( j \in [n-1] \), since \( \lambda \) is a nonzero element of \( \KK \).  Thus,
\[
0 = \left ( 1 + \sum_{j=1}^{n-1} x_j \right )^{m-1} = (1 + 0)^{m-1} = 1,
\]
we have a contradiction.  Therefore, \( x_i \neq 0 \) for any \( \bx \in V \).

For any \( i \in [n-1] \), we may write
\[
x_i^{m-1} = \frac{\left (1 + \sum_{j=1}^{n-1} x_j \right )^{m-1}}{\lambda}.
\]
Since \( x_i \) is nonzero, it must be the case that \( x_i = \zeta_{m-1}^{j_i} a \) for some \( j_i \in \{0,\ldots,m-2\} \) and each \( i \in [n-1] \), where
\[
a = \frac{1 + \sum_{j=1}^{n-1} x_j}{\lambda^{1/{m-1}}}
\]
(for any choice of the root).  Suppose that \( r_0 \) of the \( x_i \) are \( a \), \( r_1 \) of the \( x_i \) are \( \zeta_{m-1} a \), \( r_2 \) of the \( x_i \) are \( \zeta_{m-1}^2 a \), etc., where \( \br \cdot \hat{1} = n-1 \).  Then
\[
a = \frac{1 + a (\br \cdot \hat{\zeta})}{\lambda^{1/(m-1)}},
\]
so that \( a = (\lambda^{1/(m-1)} - (\br \cdot \hat{\zeta}))^{-1} \).  Since then \( a \) is uniquely determined by the choice of \( \br \), this yields exactly \( (m-1)^{n-1} \) distinct points \( \bx \in V \), one for each function \( \nu : [n-1] \rightarrow \{0,\ldots,m-2\} \), where \( \nu(j) \) represents the exponent \( j_i \) of \( \zeta_{m-1} \) chosen for the \( j \)-th coordinate of \( \bx \).  (Thus, the statement above concerning the multiplicities of the \( \bx \in V \) is verified.)

If \( \bx \in V \), then \( \zeta_{m-1}^j \bx \) is also a point of \( V \) for each \( j \in [m-2] \).  Note that
\begin{align*}
f_n(\bx) &= \lambda - \left ( 1 + \frac{\br \cdot \hat{\zeta}}{\lambda^{1/(m-1)} - (\br \cdot \hat{\zeta})} \right )^{m-1} \\
&= \lambda - \left ( \frac{\lambda^{1/(m-1)}}{\lambda^{1/(m-1)} - (\br \cdot \hat{\zeta})} \right )^{m-1} \\
&= \lambda \left [ 1 - \left ( \lambda^{1/(m-1)} - (\br \cdot \hat{\zeta}) \right )^{1-m} \right ] \\
&= \lambda \frac{ \left ( \lambda^{1/(m-1)} - (\br \cdot \hat{\zeta}) \right )^{m-1} - 1 }{\left ( \lambda^{1/(m-1)} - (\br \cdot \hat{\zeta}) \right )^{m-1}}.
\end{align*}
Then,
\begin{align*}
\prod_{j=0}^{m-2} f_n(\zeta_{m-1}^j \bx) &= \lambda^{m-1} \frac{ \prod_{j=0}^{m-2} \left [ \left ( \lambda^{1/(m-1)} - \zeta_{m-1}^j (\br \cdot \hat{\zeta}) \right )^{m-1} - 1 \right ]}{\prod_{j=0}^{m-2} \left ( \lambda^{1/(m-1)} - \zeta_{m-1}^j (\br \cdot \hat{\zeta}) \right )^{m-1}} \\
&= \lambda^{m-1} \frac{ \prod_{j=0}^{m-2} g(\zeta_{m-1}^j (\br \cdot \hat{\zeta}))}{ \prod_{j=0}^{m-2} h(\zeta_{m-1}^j (\br \cdot \hat{\zeta}))},
\end{align*}
where \( g(x) = \left ( \lambda^{1/(m-1)} - x \right )^{m-1} - 1 \) and \( h(x) = \left ( \lambda^{1/(m-1)} - x \right )^{m-1} \). All \( m-1 \) roots of \( h(x) \) are \( \lambda^{1/(m-1)} \), and its lead coefficient is \( (-1)^{m-1} \).  Therefore, by Lemma \ref{lem:rootsofunity},
\begin{align*}
\prod_{j=0}^{m-1} h(\zeta_{m-1}^j (\br \cdot \hat{\zeta})) &= (-1)^{(m-1)^2} \prod_{j=0}^{m-1} ((\br \cdot \hat{\zeta})^{m-1} - \lambda) \\
&= (-1)^{(m-1)} ((\br \cdot \hat{\zeta})^{m-1} - \lambda)^{m-1} \\
&= (\lambda - (\br \cdot \hat{\zeta})^{m-1})^{m-1}.
\end{align*}
Define \( k(y) = \lambda - (x+y)^{m-1} \).  Then \( k(y) = 0 \) implies \( y = \lambda^{1/(m-1)} \zeta_{m-1}^j - x \) for some \( 0 \leq j \leq m-2 \), so Lemma \ref{lem:rootsofunity} yields
\begin{align*}
\prod_{j=0}^{m-2} k(\zeta_{m-1}^j y) &= (-1)^{m-1} \prod_{j=0}^{m-2} (y^{m-1} - (\lambda^{1/(m-1)} \zeta_{m-1}^j - x)^{m-1}) \\
&= \prod_{j=0}^{m-2} ((\lambda^{1/(m-1)} - \zeta_{m-1}^{-j} x)^{m-1} \zeta_{m-1}^{j(m-1)}  - y^{m-1}) \\
&= \prod_{j=0}^{m-2} ((\lambda^{1/(m-1)} - \zeta_{m-1}^{-j} x)^{m-1} - y^{m-1}).
\end{align*}
Hence,
\begin{align*}
\prod_{j=0}^{m-2} g(\zeta_{m-1}^j x) &= \prod_{j=0}^{m-2} ((\lambda^{1/(m-1)} - \zeta_{m-1}^{j} x)^{m-1} - 1) \\
&= \prod_{j=0}^{m-2} ((\lambda^{1/(m-1)} - \zeta_{m-1}^{-j} x)^{m-1} - 1) \\
&= \prod_{j=0}^{m-2} k(\zeta_{m-1}^j \cdot 1) \\
&= \prod_{j=0}^{m-2} (\lambda - (x+\zeta_{m-1}^j)^{m-1}).
\end{align*}
and we may conclude that
\[
\prod_{j=0}^{m-2} f_n(\zeta_{m-1}^j \bx) = \lambda^{m-1} \frac{ \prod_{j=0}^{m-2} (\lambda - ((\br \cdot \hat{\zeta})+\zeta_{m-1}^j)^{m-1})}{(\lambda - (\br \cdot \hat{\zeta})^{m-1})^{m-1}}.
\]
Let us define, then, the function
\[
\tilde{f}_n(\bx) = \frac{\lambda \left ( \prod_{j=0}^{m-2} (\lambda - ((\br \cdot \hat{\zeta})+\zeta_{m-1}^j)^{m-1}) \right )^{1/(m-1)}}{\lambda - (\br \cdot \hat{\zeta})^{m-1}}.
\]
By the above computation, we may write
\begin{align*}
\prod_{\bx \in V} f_n(\bx) &= \prod_{\bx \in V} \tilde{f}_n(\bx) \\
&= \prod_{\br \cdot \hat{1} = n-1} \left ( \frac{\lambda \left ( \prod_{j=0}^{m-2} (\lambda - ((\br \cdot \hat{\zeta})+\zeta_{m-1}^j)^{m-1}) \right )^{1/(m-1)}}{\lambda - (\br \cdot \hat{\zeta})^{m-1}} \right )^{\mu(\br)} \\
&= \frac{ \lambda^{(m-1)^{n-1}} }{\prod_{\br \cdot \hat{1} = n-1} ( \lambda - (\br \cdot \hat{\zeta})^{m-1})^{\mu(\br)}}  \\
& \qquad \cdot \prod_{\br \cdot \hat{1} = n-1} \left ( \prod_{j=0}^{m-2} (\lambda - ((\br \cdot \hat{\zeta})+\zeta_{m-1}^j)^{m-1}) \right )^{\mu(\br)/(m-1)}.
\end{align*}
The total multiplicity of the term \( \lambda - ((\br \cdot \hat{\zeta})+\zeta_{m-1}^j) \) in the right-hand product is \( \mu(\br^\prime)/(m-1) \), where \( \br^\prime = \br + \be_j \), \( \be_j \) the elementary basis vector with nonzero coordinate \( j \).  Therefore,
\begin{align*}
\prod_{\bx \in V} f_n(\bx) &= \frac{ \lambda^{(m-1)^{n-1}} }{\prod_{\br \cdot \hat{1} = n-1} ( \lambda - (\br \cdot \hat{\zeta})^{m-1})^{\mu(\br)}}  \\
& \qquad \cdot \prod_{\br^\prime \cdot \hat{1} = n} (\lambda - (\br^\prime \cdot \hat{\zeta}))^{m-1})^{\mu(\br^\prime)/(m-1)}.
\end{align*}
Substituting this expression into (\ref{eq7}) yields
\[
\phi_{n}(\lambda) = \lambda^{(n-1) (m-1)^{n-1}} \prod_{\br \cdot \hat{1} = n} (\lambda - (\br \cdot \hat{\zeta})^{m-1})^{\mu(\br)/(m-1)},
\]
which is the desired conclusion.
\end{proof}

Our next result illustrates the geometry of the eigenvalues of \(\JJ_n^m\) as a point-set in the complex plane.  Let \( \nu_n \) denote the probability measure on \( \CC \) defined by choosing an element of the spectrum of \( \JJ_n^m \) (with multiplicity) uniformly at random.  This distribution is straightforward to describe exactly in several cases, and to describe up to total variation in all other cases.

\begin{theorem} Let \( \binom{n}{a_1,\ldots,a_k} \) denote the multinomial coefficient, which we define to be zero unless \( n = \sum_{j=1}^k a_j \) and \( a_j \in \NN \) for all \( j \in [k] \), in which case it equals \( n! / (a_1! \cdots a_k!)  \).  The distribution \( \nu_n \) is everywhere zero except as follows:
\begin{enumerate}
\item When \( m=2 \), \( \nu_n(0) = (n-1)/n \) and \( \nu_n(n) = 1/n \).
\item When \( m=3 \), \(\nu_n(k) = 2^{-n+1} \binom{n}{n + \sqrt{k}}/n\) for \(k \neq 0\) and \(\nu_n(0) = (n-1)/n + 2^{-n} \binom{n}{n/2}/n\).
\item When \( m=4 \),
    \[
    \nu_n\left ( \frac{x + y \sqrt{-3}}{2} \right ) = \frac{3^{-n+1}}{n} \binom{n}{\frac{n+a}{3}, \frac{2n-a+3b}{6}, \frac{2n-a-3b}{6}}
    \]
    for \((x,y) \in \ZZ^2 \setminus \{(0,0)\}\) if there exist \(a,b \in \ZZ\) so that \(4x = a(a^2 - 9b^2)\) and \(4y = 3b(a^2-b^2)\), and
    \[
    \nu_n(0) = \frac{n-1}{n} + \frac{3^{-n}}{n} \binom{n}{n/3, n/3, n/3}.
    \]
\item When \( m=5 \),
    \[
    \nu_n(x+yi) = \frac{4^{-n+1}}{n} \binom{n}{(n+a+b)/2} \binom{n}{(n-a+b)/2}
    \]
    for \((x,y) \in \ZZ^2 \setminus \{(0,0)\}\), if there exist \(a,b \in \ZZ\) so that \(x = a^4 - 6a^2b^2 + b^4\) and \(y = 4 ab (a^2 - b^2)\), and
    \[
    \nu_n(0) = \frac{n-1}{n} + \frac{4^{-n}}{n} \binom{n}{n/2}^2.
    \]
\item For any \(m \geq 4\),
    \[
    \sqrt{n} \nu_n \circ \psi_n - \frac{n-1}{\sqrt{n}} \delta_0 \overset{D}{\longrightarrow} \cN(0,1) \circ \rho_{m-1},
    \]
    where \( \cN(t,\sigma) \) is a standard normal distribution on \( \CC \) with mean \( t \) and standard deviation \( \sigma \), \(\psi_n(\xi) = \xi/\sqrt{n}\) for \(\xi \in \CC\), \(\delta_0\) is the probability distribution consisting of a single atom at \(0\), and \(\rho_r\) is the multivalued map which sends \(\xi \in \CC\) to all of its \(r\)-th roots, \(r \in \NN^+\).
\end{enumerate}
\end{theorem}
\begin{proof}  Note that there are \(n (m-1)^{n-1}\) eigenvalues of \(\JJ_n^m\), counted with (algebraic) multiplicity.  In the statement of Theorem \ref{thm:charpolyJJ}, the exponent of \(\lambda\) to the left of the product is \((n-1)(m-1)^{n-1}\); this accounts for a \( (n-1)(m-1)^{n-1} / (n (m-1)^{n-1}) = (n-1)/n\) fraction of the roots of the characteristic polynomial.  It follows that for, \(\xi \neq 0\), \(\nu(\xi)\) is \(1/n\) times the probability that a simple random walk whose steps are the \((m-1)\)-st roots of unity ends at some \(m-1\)-st root of \(\xi\) after \(n\) steps, and, for \(\xi = 0\), \(\nu(\xi)\) is this same quantity, plus \((n-1)/n\).
\begin{enumerate}
\item For \(m=2\), there is only one \((m-1)\)-st root of unity: \(1\).  Therefore, a simple random walk of length \(n\) with this as its only choice of step will end at \(n\) with probability \(1\).  It follows that \(\nu_n(n) = 1/n\) and \(\nu_n(0) = (n-1)/n\).
\item For \(m=3\), there are two \((m-1)\)-st roots of unity: \(1\) and \(-1\).  Therefore, a walk of length \(n\) with \(\pm 1\) as its steps will end at \(k\) if and only if it consists of \(x\) steps in the direction of \(1\) and \(y\) steps in the direction of \(-1\), where \(x-y = k\) and \(x+y = n\).  Solving for \(x\) and \(y\), one obtains \(x = (n+k)/2\) and \(y = (n-k)/2\).  The number of such walks is \(\binom{n}{(n+k)/2}\), so that the probability of ending at \(k\) is \(2^{-n} \binom{n}{(n+k)/2}\).  Since the only elements of \(\CC\) both of whose square roots are elements of \(\ZZ\) are themselves elements of \(\ZZ\), the probability that a random walk with steps \(\pm 1\) ends at an \((m-1)\)-st root of \(k \in \ZZ \setminus \{0\}\) is
    \[
    2^{-n} \left ( \binom{n}{(n+\sqrt{k})/2} + \binom{n}{(n-\sqrt{k})/2} \right ) = 2^{-n+1} \binom{n}{(n + \sqrt{k})/2}.
    \]
    and \(2^{-n} \binom{n}{n/2}\) if \(k = 0\).  Therefore, \(\nu_n(k) = 2^{-n+1} \binom{n}{n + \sqrt{k}}/n\) for \(k \neq 0\) and \(\nu_n(0) = (n-1)/n + 2^{-n} \binom{n}{n/2}/n\).
\item For \(m=4\), there are three \((m-1)\)-st roots of unity: \(S = \{1, \zeta_3, \zeta_3^2\}\), where \(\zeta = \exp(2 \pi i / 3) = (-1+i\sqrt{3})/2 \).  Therefore, a walk of length \(n\) with \(S\) as its steps will end at \(a/2 + b \sqrt{-3}/2\), \(a,b \in \ZZ\) with \(a \equiv n \pmod{2}\), if and only if it consists of \(\alpha\) steps in the direction of \(1\), \(\beta\) steps in the direction of \(\zeta\), and \(\gamma\) steps in the direction of \(\zeta^2\), where \(\alpha,\beta,\gamma \in \NN\) and
    \begin{eqnarray*}
    \alpha + \beta + \gamma & = & n \\
    \alpha - \beta/2 - \gamma/2 & = & a/2 \\
    \beta \sqrt{3}/2 - \gamma \sqrt{3}/2 & = & b \sqrt{3}/2.
    \end{eqnarray*}
    Solving for \(\alpha\), \(\beta\), and \(\gamma\) yields:
    \begin{eqnarray*}
    \alpha & = & \frac{1}{3} a + \frac{1}{3} n \\
    \beta & = & -\frac{1}{6} a + \frac{1}{2} b + \frac{1}{3} n \\
    \gamma & = & -\frac{1}{6} a - \frac{1}{2} b + \frac{1}{3} n.
    \end{eqnarray*}
    Therefore, the number of walks on the lattice generated by \(S\) that end at \(a/2 + b \sqrt{-3}/2\) is given by
    \[
    F(n,a,b) = \binom{n}{\frac{n+a}{3}, \frac{2n-a+3b}{6}, \frac{2n-a-3b}{6}}.
    \]
    Note that
    \[
    \left ( \frac{a}{2} + \frac{b \sqrt{-3}}{2} \right )^3 = \frac{a(a^2 - 9 b^2)}{8} + \frac{3\sqrt{-3}b(a^2 - b^2)}{8}.
    \]
    Since \(a \equiv b \pmod{2}\) implies that \(4 | a^2 - b^2\) and \(4 | a^2 - 9b^2\), the right-hand side is an element of \(\frac{1}{2} \ZZ[\sqrt{-3}]\).  We may conclude that the probability that a simple random walk whose steps are the cube roots of unity ends after \(n\) steps at a cube root of \(x/2 + y\sqrt{-3}/2 \neq 0\), \(x,y \in \ZZ\), is \(3 \cdot 3^{-n} F(n,a,b)\) if there exist \(a,b \in \ZZ\) so that \(4x = a(a^2 - 9b^2)\) and \(4y = 3b(a^2-b^2)\), and \(0\) otherwise.  Therefore, for \((x, y) \neq (0,0)\),
    \[
    \nu_n\left ( \frac{x + y \sqrt{-3}}{2} \right ) = \frac{3^{-n+1}}{n} \binom{n}{\frac{n+a}{3}, \frac{2n-a+3b}{6}, \frac{2n-a-3b}{6}}
    \]
    if there exist \(a,b \in \ZZ\) so that \(4x = a(a^2 - 9b^2)\) and \(4y = 3b(a^2-b^2)\), and
    \[
    \nu_n(0) = \frac{n-1}{n} + \frac{3^{-n}}{n} \binom{n}{n/3, n/3, n/3}.
    \]
\item For \(m=5\), there are four \((m-1)\)-st roots of unity: \(S = \{\pm 1, \pm i\}\).  A simple random walk of length \(n\) with \(S\) as its steps can be recast as a walk on \(\ZZ^2\) using the standard identification \(\CC \equiv \RR^2\); it is well-known (see, for example, \cite{DoyleSnell84}), that such a walk will end at \(a + bi\), \(a,b \in \ZZ\), with probability
    \[
    4^{-n} \binom{n}{(n+a+b)/2} \binom{n}{(n-a+b)/2}.
    \]
    Since \((a+bi)^4 = (a^4 - 6a^2b^2 + b^4) + 4 ab (a^2 - b^2) i \in \ZZ[i]\), we may conclude that, for \((x,y) \in \ZZ^2 \setminus \{(0,0)\}\),
    \[
    \nu_n(x+yi) = \frac{4^{-n+1}}{n} \binom{n}{(n+a+b)/2} \binom{n}{(n-a+b)/2}
    \]
    if there exist \(a,b \in \ZZ\) so that \(x = a^4 - 6a^2b^2 + b^4\) and \(y = 4 ab (a^2 - b^2)\), and
    \[
    \nu_n(0) = \frac{n-1}{n} + \frac{4^{-n}}{n} \binom{n}{n/2}^2.
    \]
\item Let \(X_r\) denote a length \(n\) simple random walk on \(\CC\) whose steps are the \(r\)-th roots of unity, \(r \geq 1\), and let \(Y_r = (\Re(X_r),\Im(X_r))\).  The covariance matrix of \(Y_r\) is given by
    \begin{align*}
    \cov(Y_r) &= \EE [\|Y_r\|_2^2] \\
    &= \left [ \begin{array}{cc} \EE [\Re(X_r)^2] & \EE[\Re(X_r)\Im(X_r)] \\ \EE[\Re(X_r)\Im(X_r)] & \EE[\Im(X_r)^2] \end{array} \right ].
    \end{align*}
    Let \(\zeta_r = \exp(2 \pi i/r)\).  Note that
    \[
    \sum_{j=0}^{r-1} \zeta_r^{2j} = \frac{\zeta_r^{2r}-1}{\zeta_r^2 - 1} = 0
    \]
    as long as the denominator is not zero, i.e., \(r (= m-1) \geq 3\).  Taking the real and imaginary parts of this equation yields
    \[
    \sum_{j=0}^{r-1} \cos(4 \pi i j / r) = \sum_{j=0}^{r-1} \sin(4 \pi i j / r) = 0.
    \]
    Therefore,
    \begin{align*}
    \EE [\Re(X_r)^2] &= \frac{1}{r} \sum_{j=0}^{r-1} \cos^2(2 \pi ij/r) \\
    &= \frac{1}{2r} \sum_{j=0}^{r-1} \cos(4 \pi ij/r) + 1\\
    &= 1/2.
    \end{align*}
    It also follows that
    \begin{align*}
    \EE [\Im(X_r)^2] &= \frac{1}{r} \sum_{j=0}^{r-1} \sin^2(2 \pi ij/r) \\
    &= \frac{1}{2r} \sum_{j=0}^{r-1} 1 - \cos(4 \pi ij/r) \\
    &= 1/2.
    \end{align*}
    Furthermore,
    \begin{align*}
    \EE [\Re(X_r)\Im(X_r)] &= \frac{1}{r} \sum_{j=0}^{r-1} \sin(2 \pi ij/r) \cos(2 \pi ij/r) \\
    &= \frac{1}{2r} \sum_{j=0}^{r-1} 2 \sin(4 \pi ij/r) \\
    &= 0.
    \end{align*}
    Therefore, \(\cov(Y_r) = \frac{1}{2} I\), and, by the Multidimensional Central Limit Theorem (see, e.g., \cite{vanderVaart98}),
    \[
    \frac{X_r}{\sqrt{n}} \overset{D}{\longrightarrow} \cN(0,1).
    \]
    It follows that
    \[
    \sqrt{n} \nu_n \circ \psi_n - \frac{n-1}{\sqrt{n}} \delta_0 \overset{D}{\longrightarrow} \cN(0,1) \circ \rho_{m-1}.
    \]
\end{enumerate}
\end{proof}

When \(m = 7\), it does not appear possible to write down a closed-form expression for \(\nu_n(\xi)\) (\cite{Zeilberger13}).  For other \(m\), however, we conjecture the following statement that strengthens the conclusion of part (5) above.

\begin{conj} For \( m = 6 \) or \( m > 7 \), \(\epsilon > 0\), and any measurable set \(E \subset \CC\),
    \[
    \lim_{n \rightarrow \infty} n^{-\epsilon} \left | \int_E n \nu_n \circ \psi_n - (n-1) \delta_0 - \sqrt{n} \cN(0,1) \circ \rho_{m-1} \right | = 0
    \]
    where \( \cN(t,\sigma) \) is a standard normal distribution on \( \CC \) with mean \( t \) and standard deviation \( \sigma \), \(\psi_n(\xi) = \xi/\sqrt{n}\) for \(\xi \in \CC\), \(\delta_0\) is the probability distribution consisting of a single atom at \(0\), and \(\rho_r\) is the multivalued map which sends \(\xi \in \CC\) to all of its \(r\)-th roots, \(r \in \NN^+\).
\end{conj}

\section{Sunflower Hypergraphs} \label{sect:sunflowers}
In previous work \cite{Cooper_Dutle_2012}, the authors computed the (normalized) spectrum of the adjacency hypermatrix of some classes of hypergraphs. Other than (possibly partial) matchings, the techniques used therein disallowed computation of the multiplicities of eigenvalues for any infinite class. Here, however, we determine the characteristic polynomial for an infinite class of \emph{sunflower} hypergraphs.

A \emph{sunflower} \( \mathcal{S}(m,q,k) \) hypergraph is defined as follows. Let \( m>0 \), and \( q, k \) satisfy \( 0<q<k. \) Let \( S \) be a set of \( q \) vertices (``seeds'') and define \( m \) disjoint sets \( \{E_i\}_{i=1}^m \) of \( k-q \) vertices each (``petals''). The edges of the hypergraph are the sets \( S\cup E_i \) for \( 1\leq i\leq m. \)

When \( k=2 \), sunflower graphs are normally referred to as \emph{stars.} In general, when \( q=k-1 \), \( \mathcal{S}(m,k-1, k) \) is a complete \( k \)-cylinder (i.e., \( k \)-partite, \( k \)-uniform hypergraph) with block sizes \( m \) and \( 1^{k-1} \), which the authors considered in \cite{Cooper_Dutle_2012}.

By using the product formula, we determine the spectrum, including multiplicities, for single-seed sunflowers of uniformity \( 3 \).

\begin{theorem} \label{sunflower_spec_mult} The characteristic polynomial for \( \mathcal{S}(n,1,3) \) is
\[\phi_{\mathcal{S}(n,1,3)}(\lambda) = \lambda^{(2n-2)4^n} \prod_{r=0}^n (\lambda^3-r)^{\binom{n}{r}3^r}.          \]
\end{theorem}

\begin{proof} Label the \( 2n+1 \) vertices of \( \mathcal{S}(n,1,3) \) as \( v_0 \) for the seed, and pairs of vertices \( \{v_{i,1}, v_{i,2}\} \) for \( i \in [n]. \)

Then the equations defining the eigenpairs for our sunflower are the following:

\[F_0 = \lambda x_0 - \sum_{i=1}^n x_{i,1}x_{i,2}\]
and for each \( i\in [n], \) the pair of equations

\[\left\{
\begin{array}{c} F_{i,1} = \lambda x_{i,1}^2- x_0x_{i,2} \\
F_{i,2} = \lambda x_{i,2}^2- x_0x_{i,1} \end{array} \right\}. \]

Consider \(   \Res( \overline{F_{1,1}}, \overline{F_{1,2}}, \ldots, \overline{F_{n,1}}, \overline{F_{n,1}}) \). Note that \( \overline{F_{i,j}} = \lambda x_{i,j}^2 \) for every \( i\in [n] \) and all \( j\in [2], \)  which is the set of equations used to define the characteristic polynomial of the hypermatrix with all zero entries. This clearly has 0 as its only eigenvalue, and so
\[  \Res( \overline{F_{1,1}}, \overline{F_{1,2}}, \ldots, \overline{F_{n,1}}, \overline{F_{n,1}}) = \lambda^{2n2^{2n-1}}.\]
As mentioned in the introduction, this is a non-zero polynomial in \( \lambda \), and so the product formula applies.
This formula gives that

\begin{align} \phi_{\mathcal{S}(n,1,3)}(\lambda)  & = \Res(F_0, F_{1,1}, F_{1,2}, \ldots, F_{n,1}, F_{n,1}) \nonumber \\
& = \Res( \overline{F_{1,1}}, \overline{F_{1,2}}, \ldots, \overline{F_{n,1}}, \overline{F_{n,1}})^2 \prod_{p\in V} f_0(p)^{m(p)} \nonumber \\
 \label{sunfcalc} & = \lambda^{2n2^{2n}}\prod_{p\in V} f_0(p)^{m(p)}.
\end{align}

We next determine the points of the variety \( V \). Fix \( i\in [n] \), and consider the pair of polynomials
\[\left\{
\begin{array}{c}
f_{i,1} = x_{i,2} -\lambda x_{i,1}^2 \\
f_{i,2} = x_{i,1} -\lambda x_{i,2}^2.
\end{array}
\right\} \]

Let \( \zeta_3 \) denote a primitive third root of unity, and choose \( x_{i,1}\in \{ 0, \frac{1}{\lambda}, \frac{\zeta_3}{\lambda},\frac{\zeta_3^2}{\lambda}\} \). It's easy to verify that this value for \( x_{i,1} \) can be uniquely extended to a zero for our two equations \( f_{i,1}, f_{i,2} \)  by setting \( x_{i,2} = \lambda x_{i,1}^2 \). It's also easy to check that \[x_{i,1}x_{i,2} = \begin{cases} 0  &\text{ if } x_{i,1}=0 \\ \frac{1}{\lambda^2} & \text{ otherwise. } \end{cases}\]

A point of our variety must be a zero for each of the \( n \) pairs of polynomials \( \{f_{i,1}, f_{i,2}\} \). Because the variables appearing in any particular pair don't appear in any of the other polynomials describing the variety, a solution can be found by combining solutions for the pairs. We saw above that we can find a unique solution for a pair of these polynomials by choosing \( x_{i,1} \) from one of the four values  \(  0, \frac{1}{\lambda}, \frac{\zeta_3}{\lambda},\frac{\zeta_3^2}{\lambda} \). Hence we can build \( 4^n = 2^{2n} \) distinct zeros, which is the most we can have in a zero dimensional variety defined by \( 2n \) quadratics. Hence all points of the variety are given by the above construction, and the mulitplicity of each point is exactly one.

 Next, we note that  \( f_0(\bx) = \lambda - \sum_{i=1}^n x_{i,1}x_{i,2}. \) As we saw above, for any point \( p\in V \), each term in this sum is either 0 or \( 1/\lambda^2 \). If we let \( r = |\{ i \mid p_{i,1} \neq 0\}| \)  denote the number of nonzero pairs of  coordinates,  we see that \[f_0(p) = \lambda-\frac{r}{\lambda^2} = \frac{\lambda^3-r}{\lambda^2}.\] For a fixed \( r \), there are \( \binom{n}{r}3^r \) such points \( p\in V. \) Thus we can compute the second factor for our resultant as

 \begin{align*} \prod_{p\in V} f_0(p)^{m(p)}
 & = \prod_{r=0}^n\left(\frac{\lambda^3-r}{\lambda^2}\right)^{\binom{n}{r}3^r} \\
 & = \lambda^{-2\cdot 4^n} \prod_{r=0}^n \left(\lambda^3-r\right)^{\binom{n}{r}3^r}.\end{align*}

  Substituting this expression into \eqref{sunfcalc} gives the claimed result.
\end{proof}

Unfortunately, we were unable to extend the result to \( \mathcal{S}(n,q,k) \) sunflowers.  Following the technique above first requires that \( q=1. \) If we attempt to use an arbitrary \( k, \) we find that there are  \( k^{k-2}+1 \) distinct solutions for the \( k-1 \) equations arising from a particular petal. Hence there are
\( (k^{k-2}+1)^n \) solutions, whereas we need \( \left((k-1)^{k-1}\right)^n \) distinct solutions to use the product formula. These two values agree precisely when \( k=3 \), making the computation above possible.

\bibliography{Research}

\def\cprime{$'$}
\begin{thebibliography}{1}

\bibitem{Cooper_Dutle_2012}
Joshua Cooper and Aaron Dutle.
\newblock Spectra of uniform hypergraphs.
\newblock {\em Linear Algebra Appl.}, 436(9):3268--3292, 2012.

\bibitem{Cox_Little_OShea_2005}
David~A. Cox, John Little, and Donald O'Shea.
\newblock {\em Using algebraic geometry}, volume 185 of {\em Graduate Texts in
  Mathematics}.
\newblock Springer, New York, second edition, 2005.

\bibitem{DoyleSnell84}
Peter~G. Doyle and J.~Laurie Snell.
\newblock {\em Random walks and electric networks}.
\newblock Carus Mathematical Monographs. Mathematical Association of America,
  Washington, DC, 1984.

\bibitem{Qi_2005}
Liqun Qi.
\newblock Eigenvalues of a real supersymmetric tensor.
\newblock {\em J. Symbolic Comput.}, 40(6):1302--1324, 2005.

\bibitem{vanderVaart98}
A.~W. van~der Vaart.
\newblock {\em Asymptotic statistics}, volume~3 of {\em Cambridge Series in
  Statistical and Probabilistic Mathematics}.
\newblock Cambridge University Press, Cambridge, 1998.

\bibitem{Zeilberger13}
Doron Zeilberger.
\newblock Personal communication, 2013.

\end{thebibliography}
\bibliographystyle{plain}


\end{document}